\documentclass[11pt,reqno]{amsart}%
\usepackage{graphicx}
\usepackage{amsmath}
\usepackage{amsfonts}
\usepackage{amssymb}
\usepackage{amsthm}
\usepackage{enumerate}
\usepackage{color}
\usepackage{epstopdf}
\usepackage{amsfonts}
\usepackage{algorithmic}
\usepackage{boxedminipage}

\usepackage{cite}
\usepackage{url}
\usepackage[margin=1in]{geometry}%
\usepackage{tikz-cd}
\usepackage{graphicx}
\usepackage{acronym}
\providecommand{\U}[1]{\protect\rule{.1in}{.1in}}

\newtheorem{theorem}{Theorem}[section]
\newtheorem{proposition}[theorem]{Proposition}
\newtheorem{lemma}[theorem]{Lemma}

\newtheorem{definition}[theorem]{Definition}

\theoremstyle{remark}

\let\O\undefined

\DeclareMathOperator{\O}{O}

\begin{document}
\title[nondegeneracy of tensor eigenvectors]{Nondegeneracy of eigenvectors and singular vector tuples of tensors}

\author{Shenglong Hu}
\address{Department of Mathematics, School of Science, Hangzhou Dianzi University, Hangzhou, 310018, China}
\email{shenglonghu@hdu.edu.cn}

\begin{abstract}
In this article, nondegeneracy of singular vector tuples, Z-eigenvectors and eigenvectors of tensors is studied. {They have found many applications in diverse areas.} The main results {are:} (i) each (Z-)eigenvector/singular vector tuple of a generic tensor is nondegenerate, and (ii) each nonzero Z-eigenvector/singular vector tuple of an orthogonally decomposable tensor is nondegenerate.
\end{abstract}
\maketitle

\section{Introduction}\label{sec:intro}
Tensors, as higher order {generalizations} of vectors and matrices, are inevitable in engineering, scientific computing as well as mathematics \cite{H-12,L-12}.
Eigenvectors and singular vector tuples of tensors have {found important applications in diverse areas} in a wide spread range since the independent seminal works of Lim \cite{L-05} and Qi \cite{Q-05}. When a numerical algorithm is designed for computing an eigenvector or a singular vector tuple, the convergence rate analysis of this algorithm {typically} involves the second order information of the eigenvector or the singular vector tuple. However, in the literature, there exists few discussions on this important topic \cite{Q-05,L-05,HHLQ-13}. As a result, {restrictive} hypotheses on eigenvectors or singular vector tuples are employed to achieve linear, superlinear or quadratic convergence of certain algorithms for computing them.

Recently, it {was} shown that all the singular vector tuples of a generic tensor are nondegenerate \cite{HL-18}, with which it is further shown that the higher order power method for computing a best rank one approximation of a given tensor converges $R$-linearly without any further assumption for a generic tensor.
{The ingredient on showing the linear convergence there is based on the fact that the quadratic approximation of the objective function over the feasible set is nonsingular locally, or more directly the \L ojasiewicz exponent of the objective function is $\frac{1}{2}$ at the critical point \cite{LP-16}.}
Therefore, this result can also be employed to any optimization methods and the corresponding local convergence rates can be established without any further assumption. It is also shown in \cite{HL-18} that any nonzero singular vector tuple of an orthogonally decomposable tensor is nondegenerate. Thus, the conclusions on convergence rate as above hold for orthogonally decomposable tensors without the generic assumption.

It sheds {light on} that the second order information analysis of eigenvectors and singular vector tuples of tensors will be important. Thus, in this article, we will first summarize the results for singular vector tuples established in \cite{HL-18}, since for which the article is presented along the clue for the linear convergence of the higher order power method. Then, we will continue the discussions to Z-eigenvectors of symmetric tensors and more general tensors, and eigenvectors of general tensors. The main derived results can be concisely summarized as
\begin{equation*}
\begin{array}{c}
\textbf{Each (Z-)eigenvector or singular vector tuple of a generic tensor is nondegenerate.}
\end{array}
\end{equation*}
Moreover, we also show that every nonzero Z-eigenvector of a symmetric orthogonally decomposable tensor is nondegenerate.

The rest of this article is organized as follows. Some preliminaries are presented in Section~\ref{sec:notation}. Results on singular vector tuples are summarized in Section~\ref{sec:singular}. The new {contributions start} from Section~\ref{sec:z-eigenvector}, which is for Z-eigenvectors of {real} symmetric tensors and general {(complex)} tensors.  Section~\ref{sec:symmetry} is for symmetric tensors, Section~\ref{sec:orthogonal} is for symmetric orthogonally decomposable tensors, and Section~\ref{sec:z-eigenvector-nsym} is for the general case from an algebraic perspective. In particular, ``generic versions" of results in Section~\ref{sec:singular} and Section~\ref{sec:symmetry} are established. This is in the Zariski topology which is weaker than the usual Euclidean topology as that in \cite{HL-18}, which gives an ``almost all version". The eigenvectors of tensors are discussed in Section~\ref{sec:eigenvector}. Some final remarks are given in Section~\ref{sec:final}.

\section{Preliminaries}\label{sec:notation}
Let $n$ and $k\geq 3$ be positive integers and $\mathbb R$ ({resp. }$\mathbb C$) be the field of real ({resp. }complex) numbers.
Throughout this paper, $\|\cdot\|$ is reserved for the Euclidean norm of a vector.
Let $\mathbb S^{n-1}\subset\mathbb R^n$ {be the unit} sphere in $\mathbb R^n$. Given positive integers $k\geq 3$ and $n_1,\dots,n_k$, $\mathbb R^{n_1}\otimes\dots\otimes\mathbb R^{n_k}$ is the space of $k$th order tensors with dimension $n_1\times\dots\times n_k$. Let $\otimes^k\mathbb R^n:=\mathbb R^n\otimes\dots\otimes\mathbb R^n$ ($k$ copies) be the space of $k$th order tensors of dimension $n\times\dots\times n$ with entries in $\mathbb R$. $\otimes^k\mathbb C^n$ is defined similarly.

Given a vector $\mathbf x\in\mathbb C^n$ with entries $x_i$'s, $\mathbf x^{\otimes k}$ represents the \textit{decomposable tensor} defined by $\mathbf x$, which is a symmetric tensor with entries being
\[
(\mathbf x^{\otimes k})_{i_1\dots i_k}=x_{i_1}\cdots x_{i_k}\ \text{for all }i_j\in\{1,\dots,n\}\ \text{and }j\in\{1,\dots,k\}.
\]
The mapping $\mathbf x\mapsto\mathbf x^{\otimes k}$ is well-known as \textit{Veronese mapping} or \textit{Veronese embedding} \cite{H-77}.
Given a tensor $\mathcal A\in\otimes^k\mathbb C^n$ with entries $a_{i_1\dots i_k}$'s and a vector $\mathbf x\in\mathbb C^n$, $\mathcal A\mathbf x^{k-2}$ is defined as a matrix in $\mathbb C^{n\times n}$ with its $(i,j)$th component being $\sum_{i_3,\dots,i_k=1}^na_{iji_3\dots i_k}x_{i_3}\cdots x_{i_k}$ for all $i,j\in\{1,\dots,n\}$. Likewise, $\mathcal A\mathbf x^{k-1}$ is defined as a vector in $\mathbb C^n$ via $(\mathcal A\mathbf x^{k-2})\mathbf x$.

Given a \textit{block vector}
\[
\mathbf x:=(\mathbf x_1,\dots,\mathbf x_k)\in\mathbb R^{n_1}\times\dots\times\mathbb R^{n_k}\simeq\mathbb R^{n_1+\dots+n_k}\ \text{with }\mathbf x_i\in\mathbb R^{n_i}\ \text{for all }i=1,\dots,k,
\]
we define a mapping $\tau : \mathbb R^{n_1}\times\dots\times\mathbb R^{n_k}\rightarrow\mathbb R^{n_1}\otimes\dots\otimes\mathbb R^{n_k}$ as the {decomposable tensor} with order $k$ defined by $\{\mathbf x_1,\dots,\mathbf x_k\}$, that is,
\begin{equation}\label{eq:segre}
\tau(\mathbf x)=\mathbf x_1\otimes\dots\otimes\mathbf x_k.
\end{equation}
This mapping is well-known as \textit{Segre mapping} or \textit{Segre embedding} \cite{H-77}.
Given two tensors $\mathcal A,\mathcal B\in \mathbb R^{n_1}\otimes\dots\otimes\mathbb R^{n_k}$ with order $k$ and the entries being indexed as $a_{i_1\dots i_k}$ and $b_{i_1\dots i_k}$ respectively, the  inner product is defined as
\[
\langle\mathcal A,\mathcal B\rangle:=\sum_{i_1=1}^{n_1}\dots\sum_{i_k=1}^{n_k}a_{i_1\dots i_k}b_{i_1\dots i_k},
\]
with the corresponding induced norm given by
$\|\mathcal A\|_{\operatorname{HS}}:=\sqrt{\langle\mathcal A,\mathcal A\rangle}$.
This norm is a generalization of the matrix Frobenius norm and is termed as the \textit{Hilbert-Schmidt norm}.

{If $f : M\rightarrow \mathbb R$ is a smooth function over a smooth manifold $M$, a \textit{critical point} of $f$ on $M$ is a point $\mathbf x\in M$ such that the Riemannian gradient of $f$ at $\mathbf x$ vanishes and a \textit{nondegenerate critical point} is a critical point $\mathbf x$ of $f$ such that the Riemannian Hessian of $f$ at $\mathbf x$ is a nonsingular linear mapping from the tangent space of $M$ at $\mathbf x$ to itself \cite{BT-82}.}

{Let $\mathbb F$ be a field. In this article, it can be $\mathbb R$, the field of real numbers, or $\mathbb C$, the field of complex numbers; the exact meaning will be clear from the context. We say a property is ``generic" in a space $\mathbb F^n$, if there exists a proper closed subset $X\subset\mathbb F^n$ in the Zariski topology such that this property holds outside $X$ \cite{H-77}. Note that $\mathbb R^n$, as a subset of $\mathbb C^n$, is dense in the Zariski topology.}
\section{Singular Vector Tuples}\label{sec:singular}
{  We first give the definitions of singular vector tuples. 
\begin{definition}\label{def:singular}
Given a tensor $\mathcal A\in\mathbb R^{n_1}\otimes\dots\otimes\mathbb R^{n_k}$, a vector tuple $\mathbf x=(\mathbf x_1,\dots,\mathbf x_k)\in \mathbb S:=\mathbb S^{n_1-1}\times\dots\times\mathbb S^{n_k-1}$ is called a (real) \textit{singular vector tuple} of $\mathcal A$ if it is a critical point of the smooth function $G(\mathbf x):=\langle\mathcal A,\tau(\mathbf x)\rangle$ on the joint sphere $\mathbb S$. The value of $G$ at a singular vector tuple is called a \textit{singular value}. The corresponding vectors $\{\mathbf x_1,\dots,\mathbf x_k\}$ are called \textit{singular vectors}.
\end{definition}
The definitions of singular values/vectors were proposed by Lim \cite{L-05}.
It is easy to see from the definition that
\[
\sigma=\langle\mathcal A,\tau(\mathbf x)\rangle
\]
for a singular vector tuple $\mathbf x$ and the corresponding singular value $\sigma$. If $\sigma\neq 0$, then we call $\mathbf x$ a \textit{nonzero singular vector tuple}. For a generic tensor (e.g., tensors with nonzero hyperdeterminant), all its singular vector tuples are nonzero \cite{GKZ-94}.

\begin{definition}[Nondegenerate Singular Vector Tuples]\label{def:nongenerate}
We say a singular vector tuple $\mathbf x$ of $\mathcal A$ is \textit{nondegenerate} if $\mathbf x$ is a nondegenerate critical point of $G$ on $\mathbb S$.
\end{definition}

A main result in \cite{HL-18} is the following theorem (cf.\ \cite[Theorem~5.3]{HL-18}).
\begin{theorem}[Almost All {Nondegeneracy}]\label{thm:critical}
For almost all tensors in $\mathbb R^{n_1}\otimes\dots\otimes\mathbb R^{n_k}$, each of its singular vector tuples is nondegenerate.
\end{theorem}

Orthogonally decomposable tensors form a very important subclass of tensors \cite{F-92,K-01,CJ-10}.
The following result can be found in \cite[Proposition~6.5]{HL-18}.
\begin{theorem}[Orthogonally Decomposable Tensors]\label{thm:orthogonal}
If a tensor $\mathcal A\in\mathbb R^{n_1}\otimes\dots\otimes\mathbb R^{n_k}$ is orthogonally decomposable, i.e.,
\begin{equation}\label{eq:orthogonal}
\mathcal A=\sum_{i=1}^r\lambda_i\mathbf u^{(1)}_i\otimes\dots\otimes\mathbf u^{(k)}_i
\end{equation}
for some orthonormal matrices $[\mathbf u^{(j)}_1\ \cdots\ \mathbf u^{(j)}_r]\in\mathbb R^{n_j\times r}$ for all $j\in\{1,\dots,k\}$, then each of its nonzero singular vector tuples is nondegenerate.
\end{theorem}

\section{Z-Eigenvectors}\label{sec:z-eigenvector}
In the following, we {recall} the definitions of Z-eigenvalues and Z-eigenvectors. Let $\operatorname{S}(\otimes^k\mathbb R^{n})\subset \otimes^k\mathbb R^{n}$ be the subspace of symmetric tensors in $\otimes^k\mathbb R^{n}$. 
\begin{definition}\label{def:z-eigenvalue}
Given a tensor $\mathcal A\in\operatorname{S}(\otimes^k\mathbb R^{n})$, a vector $\mathbf x\in \mathbb S^{n-1}$ is called a {\textit{Z-eigenvector}} of $\mathcal A$ if it is a critical point of the smooth function $S(\mathbf x):=\langle\mathcal A,\mathbf x^{\otimes k}\rangle$ on the sphere $\mathbb S^{n-1}$. The value of $S$ at a Z-eigenvector is called a \textit{Z-eigenvalue}.
\end{definition}
The definitions of Z-eigenvalues and Z-eigenvectors were proposed by Qi \cite{Q-05}. The prefix ``Z" is addressed to memorize Professor Shuzi Zhou in Hunan University \cite{Q-05}.
Given a tensor $\mathcal A\in\operatorname{S}(\otimes^k\mathbb R^{n})$, we note that, by using Lagrange multiplier method, the definition entails that
a {Z-eigenvector} $\mathbf x\in\mathbb S^{n-1}$ of $\mathcal A$
and  a corresponding {Z-eigenvalue} $\lambda$ satisfy (cf.\ \cite{B-99})
\begin{equation}\label{eq:singular}
\mathcal A\mathbf x^{k-1}=\lambda\mathbf x.
\end{equation}
It is easy to see from \eqref{eq:singular} that
\[
\lambda=\langle\mathcal A,\mathbf x^{\otimes k}\rangle
\]
for a Z-eigenvector $\mathbf x$ and the corresponding Z-eigenvalue $\lambda$.

\begin{definition}[Nondegenerate Z-Eigenvectors]\label{def:nongenerate}
We say a Z-eigenvector $\mathbf x$ of $\mathcal A$ is \textit{nondegenerate} if $\mathbf x$ is a nondegenerate critical point of $S$ on $\mathbb S^{n-1}$.
\end{definition}

{Sections~\ref{sec:symmetry} and \ref{sec:orthogonal} are established in the similar spirit as that in \cite{HL-18}. However, \cite{HL-18} is for singular vector tuples of nonsymmetric tensors, while Z-eigenvectors for symmetric tensors are studied here. The results in \cite{HL-18} cannot be applied here directly. }
\subsection{The Symmetric Case}\label{sec:symmetry}
Let $M$ be a smooth manifold, a function $f : M\rightarrow \mathbb R$ is called a \textit{Morse function} if each critical point of $f$ on $M$ is nondegenerate. The following result on existence of Morse functions is well-known, see for example \cite[Proposition~17.18]{BT-82}.
\begin{lemma}[Morse Functions]\label{lem:generic-morse}
Let $M$ be a manifold of dimension $m$ in $\mathbb R^n$. For almost all $\mathbf a:=(a_1,\dots,a_n)^\mathsf{T}\in\mathbb R^n$, the function
\[
f(\mathbf x)=a_1x_1+\dots+a_nx_n
\]
is a Morse function on $M$.
\end{lemma}

We will also need the following proposition on critical points of functions over two diffeomorphic smooth manifolds. Recall that two smooth manifolds $M_1$ and $M_2$ are called \textit{locally diffeomorphic} if there is a mapping $\phi : M_1\rightarrow M_2$ such that for each point $\mathbf x\in M_1$ there exist a neighborhood $U\subseteq M_1$ of $\mathbf x$ and a neighborhood $V\subseteq M_2$ of $\phi(\mathbf x)$ so that the restriction mapping $\phi : U\rightarrow V$ is a diffeomorphism \cite{dC-92}. In this case, the corresponding $\phi$ is called a \textit{local diffeomorphism} between $M_1$ and $M_2$. The following result can be found in \cite{HL-18}.

\begin{proposition}\label{prop:critical}
Let $M_1\subseteq\mathbb R^{n_1}$ and $M_2\subset\mathbb R^{n_2}$ be two locally diffeomorphic smooth manifolds of the same dimension $m\leq\min\{n_1,n_2\}$ and let
$\phi : M_1\rightarrow M_2$ be the corresponding local diffeomorphism.
Let $f : M_2\rightarrow \mathbb R$ be a smooth function. Then $\mathbf x\in M_1$ is a (nondegenerate) critical point of $f\circ \phi$ on $M_1$ if and only if $\phi(\mathbf x)$ is a (nondegenerate) critical point of $f$ on $M_2$.
\end{proposition}

Now, we are in the position to present one of our main results, showing that the function $S$ given in Definition~\ref{def:z-eigenvalue} is a Morse function on the sphere for almost all tensors.
\begin{theorem}[Almost All Nondegenerate Z-Eigenvectors]\label{thm:z-critical}
For almost all tensors in $\operatorname{S}(\otimes^k\mathbb R^n)$, each of its Z-eigenvector is nondegenerate.
\end{theorem}

\begin{proof}
Let
\[
M:=\{\mathcal A\in\operatorname{S}(\otimes^k\mathbb R^n)\mid \mathcal A=\mathbf x^{\otimes k}\ \text{with } \mathbf x\in\mathbb S^{n-1}\}
\]
be the image of the Veronese mapping restricted on the sphere $\mathbb S^{n-1}$ \cite{H-77}.
Recall that
\[
S(\mathbf x)=\langle\mathcal A,\mathbf x^{\otimes k}\rangle.
\]

Let $\psi : \mathbb S^{n-1}\rightarrow M$ be the Veronese mapping from the sphere to the manifold $M$, then we have
\[
S(\mathbf x)=(\tilde S\circ \psi)(\mathbf x),
\]
where $\tilde{S}(\mathcal U)=\langle\mathcal A,\mathcal U\rangle$ for all $\mathcal{U} \in \operatorname{S}(\otimes^k\mathbb R^n)$.
Note that {only} the independent entries of $\mathcal A$ and $\mathcal U$ are essential in the linear functional $\tilde{S}(\mathcal U)$. Thus, from Lemma~\ref{lem:generic-morse}, we see that $\tilde{S}$
is a Morse function over the manifold $M$ for almost all $\mathcal A\in \operatorname{S}(\otimes^k\mathbb R^n)$. {If one can show that
$\psi$ is a surjective local diffeomorphism from $\mathbb S^{n-1}$ to $M$, then the conclusion follows from Proposition~\ref{prop:critical}. The diffeomorphism fact will be proven in the sequel.} Actually, $S(\mathbf x)$ is then a Morse function on $\mathbb S^{n-1}$ and hence each Z-eigenvector is nondegenerate for almost all $\mathcal A$.

We now justify our claim that $M$ is a smooth manifold which is {locally diffeomorphic} to the sphere $\mathbb S^{n-1}$ via the Veronese mapping $\psi$. Given a point $\mathbf x\in\mathbb S^{n-1}$, let
\[
j_*\in\operatorname{argmax}\{|x_j|\mid j\in\{1,\dots,n\}\}.
\]
Obviously,
\[
|x_{j_*}|\geq \frac{1}{\sqrt{n}}.
\]
If we take $\epsilon<\frac{1}{2\sqrt{n}}$, then for every
\begin{equation}\label{eq:neig}
\mathbf y\in U:=\mathbb S^{n-1}\cap\{\mathbf w\mid\|\mathbf w-\mathbf x\|\leq\epsilon\},
\end{equation}
we have
\begin{equation}\label{eq:sign}
\operatorname{sign}([\psi(\mathbf y)]_{j_*\dots j_*})=\operatorname{sign}([\psi(\mathbf x)]_{j_*\dots j_*}),
\end{equation}
since $y_{j_*}$'s have constant sign over the neighborhood $U$ given as \eqref{eq:neig}.
In the following, we will show that $\psi : U\rightarrow M$ is a local diffeomorphism from $U$ to $V:=\psi(U)$.
To see this, let
\[
T(j_*):=\{\mathcal A\in \operatorname{S}(\otimes^k\mathbb R^n)\mid a_{j_*\dots j_*}\neq 0\}.
\]
Then, we have a smooth mapping $\phi_{j_*} : V\subset T(j_*)\cap M\rightarrow U\subset \mathbb R^{n}$ from $V$ to $U$ as
\[
\phi_{j_*}(\mathcal A) = \mathbf y
\]
where
\begin{equation}\label{eq:y}
\mathbf y:=\kappa\frac{(a_{1j_*\dots j_*},\dots,a_{nj_*\dots j_*})^\mathsf{T}}{\|(a_{1j_*\dots j_*},\dots,a_{nj_*\dots j_*})\|},
\end{equation}
and $\kappa\in\{-1,1\}$ is a constant such that
\[
\phi_{j_*}(\psi(\mathbf x)) = \mathbf x.
\]
It follows from \eqref{eq:sign} and \eqref{eq:y} that $\phi_{j_*}\circ\psi$ is the identity over $U$. Moreover, a direct verification shows that $\psi\circ\phi_{j_*}$ equals the identity mapping over $V$. So, we see that $\psi$ is a local diffeomorphism.
It is well-known that the unit sphere is a smooth manifold. Thus, $M$ is a smooth manifold which is locally diffeomorphic to the unit sphere which is of dimension $n-1$.
\end{proof}

\subsection{Orthogonally Decomposable Tensors}\label{sec:orthogonal}
In the following, we present a method for analyzing the nondegeneracy of Z-eigenvectors for a {given} tensor, other than a tensor in general position. It follows the approach introduced in \cite{HL-18}.

We consider the following system of polynomial equations for a given tensor $\mathcal A\in\operatorname{S}(\otimes^k\mathbb R^{n})$
\begin{equation}\label{eq:z-eigenvector}
T(\mathbf x):=\mathcal A\mathbf x^{k-1}-\langle\mathcal A,\mathbf x^{\otimes k}\rangle\mathbf x=\mathbf 0.
\end{equation}
We call a Z-eigenvector $\mathbf x$ of a given tensor $\mathcal A$ a \textit{nonzero Z-eigenvector} if the corresponding Z-eigenvalue $\lambda=\langle\mathcal A,\mathbf x^{\otimes k}\rangle$ is nonzero.
\begin{proposition}[Nonzero Z-eigenvector]\label{prop:sym-manifold}
Given a tensor $\mathcal A\in\operatorname{S}(\otimes^k\mathbb R^{n})$, a nonzero Z-eigenvector $\mathbf x$ is nondegenerate if and only if $\mathbf x$ is a nonsingular solution of $T(\mathbf x)=\mathbf 0$.
\end{proposition}
\begin{proof}
{Let $\operatorname{T}_{\mathbb S^{n-1}}(\mathbf x)$ be the tangent space of $\mathbf x$ on $\mathbb S^{n-1}$.}
A direct calculation shows that the manifold Hessian $\operatorname{Hess}(S)(\mathbf x)$ of $S=\langle\mathcal A,\mathbf x^{\otimes k}\rangle$ at a Z-eigenvector $\mathbf x$ is given by the formula (cf.\ \cite{EAS-98})
\begin{equation}\label{eq:sym-hessian}
\langle \Delta^{(1)}, \operatorname{Hess}(S)(\mathbf x)\Delta^{(2)}\rangle=\langle\Delta^{(1)},(k(k-1)\mathcal A\mathbf x^{k-2}-k\lambda I)\Delta^{(2)}\rangle,
\end{equation}
for any two tangent vectors $\Delta^{(1)},\Delta^{(2)}\in \operatorname{T}_{\mathbb S^{n-1}}(\mathbf x)$, and in where $\lambda=S(\mathbf x)$.
Let $\operatorname{St}(n-1,n)$ be the Stiefel manifold consisting of $n\times (n-1)$ matrices with orthonormal columns, and $P\in\operatorname{St}(n-1,n)$ be such that $P^\mathsf{T}\mathbf x=\mathbf 0$.
Then the columns of $P$ form a basis for $\operatorname{T}_{\mathbb S^{n-1}}(\mathbf x)$ in the underlying Euclidean space $\mathbb R^n$.
Therefore, $\operatorname{Hess}(S)(\mathbf x)$ is singular  if and only if there exist nonzero $\mathbf y\in\mathbb R^{n-1}$ and $\alpha\in\mathbb R$ such that
\[
\operatorname{Hess}(S)(\mathbf x)P\mathbf y=\alpha\mathbf x.
\]
Note that the manifold Hessian of $S$ is a linear operator from the tangent space $\operatorname{T}_{\mathbb S^{n-1}}(\mathbf x)$ of the manifold to itself \cite{EAS-98}.
So, $\operatorname{Hess}(S)(\mathbf x)$ is singular  if and only if there exist nonzero $\mathbf y\in\mathbb R^{n-1}$ such that
\[
\operatorname{Hess}(S)(\mathbf x)P\mathbf y= \mathbf{0}.
\]

Now, suppose that $\mathbf x$ is a nonsingular solution of $T(\mathbf x)=\mathbf 0$. Then the following matrix $\nabla_{\mathbf x}T(\mathbf x)$ is nonsingular,
\[
\nabla_{\mathbf x}T(\mathbf x):=(k-1)\mathcal A\mathbf x^{k-2}-\langle\mathcal A,\mathbf x^{\otimes k}\rangle I-k(\mathcal A\mathbf x^{k-1})\mathbf x^\mathsf{T}=(k-1)\mathcal A\mathbf x^{k-2}-\lambda I-k\lambda \mathbf x\mathbf x^\mathsf{T},
\]
and so,
for any nonzero $\mathbf y\in\mathbb R^{n-1}$, we have
\[
\mathbf 0\neq \nabla_{\mathbf x}T(\mathbf x) P\mathbf y=\frac{1}{k}\operatorname{Hess}(S)(\mathbf x)P\mathbf y.
\]
This implies that the Hessian $\operatorname{Hess}(S)(\mathbf x)$ is a nonsingular linear operator from the tangent space to itself. Thus, $\mathbf x$ is a nondegenerate Z-eigenvector.

Conversely, suppose that $\operatorname{Hess}(S)(\mathbf x)$ is a nonsingular linear operator from the tangent space to itself. We proceed by the method of contradiction and assume that  $\nabla_{\mathbf x} T(\mathbf x)$ is a singular matrix, i.e., $\nabla_{\mathbf x} T(\mathbf x) \mathbf z=\mathbf 0$ for some $\mathbf z\in \mathbb R^{n}$ with $\mathbf z \neq \mathbf 0$.
Write $\mathbf z=\alpha\mathbf x+\beta\mathbf u$ as an orthogonal decomposition, where
$\alpha,\beta\in\mathbb R$ and $\mathbf u^\mathsf{T}\mathbf x=0$. Then, we have
\[
\nabla_{\mathbf x} T(\mathbf x)(\alpha\mathbf x+\beta\mathbf u)=\mathbf 0.
\]
On the other hand, a direct calculation shows that
\[
\nabla_{\mathbf x} T(\mathbf x)(\alpha\mathbf x+\beta\mathbf u)={-2\lambda\alpha\mathbf x+\frac{1}{k}\operatorname{Hess}(S)(\mathbf x)(\beta\mathbf u).}
\]
Since $\operatorname{Hess}(S)(\mathbf x)$ maps a tangent vector into the tangent space, if the Z-eigenvalue $\lambda\neq 0$, we must have both
\[
\alpha\mathbf x=\mathbf 0\ \text{and }\operatorname{Hess}(S)(\mathbf x)(\beta\mathbf u)=\mathbf 0.
\]
This, together with the nonsingularity of $\operatorname{Hess}(S)(\mathbf x)$, implies that $\beta \mathbf u = \gamma \mathbf x$ for some $\gamma \in \mathbb{R}$. It follows that $\beta \mathbf u=\mathbf 0$ because it is an orthogonal decomposition of $\mathbf z$. This contradicts the fact that $\mathbf z \neq \mathbf 0$, and so, the conclusion follows.
\end{proof}

The merit of Proposition~\ref{prop:sym-manifold} is that it transforms a geometric object to an algebraic one. It is a starting point for the further discussion for tensors with complex entries in the following Section~\ref{sec:z-eigenvector-nsym}.

\begin{theorem}[Symmetric Orthogonally Decomposable Tensors]\label{thm:sym-orthogonal}
Let $k\geq 3$.
If a tensor $\mathcal A\in\operatorname{S}(\otimes^k\mathbb R^{n})$ is orthogonally decomposable, i.e.,
\begin{equation}\label{eq:sym-orthogonal}
\mathcal A=\sum_{i=1}^r\lambda_i\mathbf u_i^{\otimes k}
\end{equation}
for an orthonormal matrix $[\mathbf u_1\ \dots\ \mathbf u_r]\in\mathbb R^{n\times r}$, then each of its nonzero Z-eigenvectors is nondegenerate.
\end{theorem}

\begin{proof}
Given a tensor $\mathcal A\in\operatorname{S}(\otimes^k\mathbb R^{n})$ and an orthogonal matrix $U\in\O(n)$, we can define an action $U\cdot\mathcal A$, which is also a tensor in $\operatorname{S}(\otimes^k\mathbb R^{n})$, {component-wisely via}
\[
(U\cdot\mathcal A)_{i_1\dots i_k}=\sum_{j_1,\dots,j_k=1}^nu_{i_1j_1}\dots u_{i_kj_k}a_{j_1\dots j_k}\ \text{for all }i_1,\dots,i_k\in\{1,\dots,n\}.
\]

Let $\mathcal B:=U\cdot\mathcal A$. It is a direct calculation to see that $\mathbf x$ is a (nonzero) Z-eigenvector of $\mathcal A$ if and only if $U\mathbf x$ is a (nonzero) Z-eigenvector of $\mathcal B$. Moreover, if let $T_{\mathcal A}(\mathbf x)$ be the corresponding system of equations as \eqref{eq:z-eigenvector} for the tensor $\mathcal A$ and $T_{\mathcal B}(\mathbf x)$ for the tensor $\mathcal B$, then we have
\[
\nabla T_{\mathcal B}(\mathbf x)=U(\nabla T_{\mathcal A}(U^\mathsf{T}\mathbf x))U^\mathsf{T}.
\]
Therefore, $\mathbf x$ is a nondegenerate Z-eigenvector of $\mathcal A$ if and only if $U\mathbf x$ is a nondegenerate Z-eigenvector of $\mathcal B$.

With this orthogonal action and the equivalence, we can assume without loss of generality that the tensor $\mathcal A$ is a diagonal tensor with the first $r$ nonzero diagonal elements being $\lambda_1,\dots,\lambda_r$, and the rest diagonal elements being zero.

It is easy after a direct calculation to see that each nonzero Z-eigenpair (a nonzero Z-eigenvalue together with a Z-eigenvector) of $\mathcal{A}$ is of the form: for all $s=1,\dots,r$,
\begin{equation}\label{eq:singular-vector}
(\lambda,\mathbf x)=\Bigg(\operatorname{sign}(\Lambda_s)\bigg(\frac{1}{\sum_{i\in\Lambda_s }\lambda_i^{\frac{-2}{k-2}}}\bigg)^{\frac{k-2}{2}},P\mathbf w\Bigg),
\end{equation}
where $\Lambda_s\subseteq\{1,\dots,k\}$ is a subset of cardinality $s\leq k$ such that
\[
\operatorname{sign}(\lambda_t)\ \text{is constant for all }t\in \Lambda_s\ {\text{if }k\ \text{is even}},
\]
and
\[
\operatorname{sign}(\Lambda_s):=\begin{cases}\operatorname{sign}(\lambda_t)\ \text{for some }t\in \Lambda_s& \text{if }k\ \text{is even},\\ 1\ \text{or }-1&\text{if }k\ \text{is odd},\end{cases}
\]
the vector $\mathbf w$ is
\[
\mathbf w=\bigg(\frac{1}{\sum_{i\in\Lambda_s }\lambda_i^{\frac{-2}{k-2}}}\bigg)^{\frac{1}{2}}\big(|\lambda_1|^{\frac{-1}{k-2}},\dots,
|\lambda_{r}|^{\frac{-1}{k-2}}\big)^\mathsf{T},
\]
and the matrix $P$ satisfy the following property $(\operatorname{P})$:
\begin{itemize}
\item [($\operatorname{P}$) \ \ \ ] $P\in\mathbb R^{n\times r}$ is a diagonal matrix with the $(j,j)$-th diagonal element satisfying
\[
\left\{\begin{array}{ll}
 p_{jj} \in \{-1,1\}, & \mbox{ if } j \in \Lambda_s, \\
 p_{jj} =0, & \mbox{ otherwise}
       \end{array}
  \right.
\]{
such that
\[
\operatorname{sign}(p_{jj})\operatorname{sign}(\lambda_t)=\operatorname{sign}(\Lambda_s)\ \text{for all }t\in \Lambda_s\ \text{if }k\ \text{is odd}.
\]}
\end{itemize}
We can have a count on the total number of nonzero Z-eigenvectors of $\mathcal A$ as that in \cite{HL-18}.

Let
\[
\sigma:=\bigg(\frac{1}{\sum_{i\in\Lambda_s }\lambda_i^{\frac{-2}{k-2}}}\bigg)^{\frac{1}{2}}.
\]
In the following, we derive the nondegeneracy of a nonzero Z-eigenvector by using Proposition~\ref{prop:sym-manifold}.
Suppose without loss of generality that $\Lambda_s=\{1,\dots,s\}$ for some $s\leq k$ {and each nonzero component of the eigenvector is positive}. Then we have $\mathbf x=(\mathbf z^\mathsf{T},\mathbf 0)^\mathsf{T}$ with
\[
\mathbf z:=\sigma \big(|\lambda_1|^{\frac{-1}{k-2}},\dots,
|\lambda_{s}|^{\frac{-1}{k-2}}\big)^\mathsf{T},
\]
and
\[
\mathcal A\mathbf x^{k-2}=\begin{bmatrix}\lambda I& 0\\  0&0\end{bmatrix},
\]
where the size of the identity matrix $I$ is $s\times s$.
A direct calculation shows that
\[
\nabla_{\mathbf x} T(\mathbf x)=\begin{bmatrix}(k-2)\lambda I-k\lambda\mathbf z\mathbf z^\mathsf{T}&  0\\0 &-\lambda I\end{bmatrix},
\]
which is a nonsingular matrix.
The conclusion then follows from Proposition~\ref{prop:sym-manifold}.
\end{proof}

In the proof of Theorem~\ref{thm:sym-orthogonal}, the Z-eigenvectors are characterized for a symmetric tensor. This is also derived in \cite{R-16} by using algebraic geometry tools. Besides our concern is nondegeneracy here, our derivation is more elementary. For the singular vector tuple case, we refer to \cite{HL-18,RS-16}.
It follows from \cite{HHLQ-13} that for a generic tensor (e.g., a tensor with nonzero determinant) there exists only nonzero Z-eigenvectors.
\subsection{Generalizations of Z-eigenvectors}\label{sec:z-eigenvector-nsym}
Z-eigenvectors were defined for a general nonsymmetric tensor as well \cite{Q-05,L-05}.
In this section, we study the general case with nonsymmetric tensors.
To that end, a general concept is recalled.
A nonzero solution of the following system
\[
\mathcal A\mathbf x^{k-1}\wedge\mathbf x=\mathbf 0
\]
is called a \textit{E-eigenvector}, which can be complex \cite{HQ-14}. E-eigenvectors can only be determined up to scaling, and thus they are actually equivalence classes \cite{CS-13}. It is easy to see that Z-eigenvectors are normalized real E-eigenvectors for a real tensor.

In the following, we will continue our discussion as Section~\ref{sec:orthogonal}.  We consider the following system of polynomial equations for a given tensor $\mathcal A\in \otimes^k\mathbb C^{n}$
\begin{equation}\label{eq:z-eigenvector-nsym}
T(\mathbf x):=\mathcal A\mathbf x^{k-1}-\langle\mathcal A,\mathbf x^{\otimes k}\rangle\mathbf x=\mathbf 0.
\end{equation}

For a \textit{nonsingular tensor} (a tensor with nonzero determinant, which is a generic property, cf. \cite{HQ-14}), E-eigenvectors can be completely characterized by solutions of \eqref{eq:z-eigenvector-nsym} \cite{HQ-14}.

In view of Proposition~\ref{prop:sym-manifold}, a nondegenerate Z-eigenvector for a general $\mathcal A\in \otimes^k\mathbb R^{n}$ should be defined as a nonsingular real solution of \eqref{eq:z-eigenvector-nsym}. More generally, we have the following definition.
\begin{definition}[Nondegenerate E-eigenvector]\label{def:e-eigenvector}
An E-eigenvector of a given tensor $\mathcal A$ is nondegenerate if it is a nonsingular solution of the system \eqref{eq:z-eigenvector-nsym}.
\end{definition}

Let us start by recalling the following lemma which can be found in \cite[Theorem~7.1.1]{SW-05}.
\begin{lemma}[Parametric Polynomial Systems]\label{lem:polynomial}
Let $G(\mathbf x;\mathbf y) : \mathbb C^s\times\mathbb C^t\rightarrow\mathbb C^s$ be a system of polynomials in $s$ variables collected in $\mathbf x$ and $t$ parameters collected in $\mathbf y$. Denote by $\mathcal N(\mathbf y)$ the number of nonsingular solutions of $G(\mathbf x,\mathbf y)=\mathbf 0$ as a function of $\mathbf y$, i.e.,
\[
\mathcal N(\mathbf y):=\#\{\mathbf x\in\mathbb C^s\mid G(\mathbf x,\mathbf y)=\mathbf 0, \ \operatorname{det}(\nabla_{\mathbf x}G(\mathbf x,\mathbf y))\neq 0\}.
\]
Then we have
\begin{enumerate}
\item $\mathcal N(\mathbf y)$ is finite, and the same constant $\mathcal N$ for all $\mathbf y$ in a nonempty Zariski open subset of $\mathbb C^t$.
\item $\mathcal N(\mathbf y)\leq\mathcal N$ for all $\mathbf y\in\mathbb C^t$.
\end{enumerate}
\end{lemma}

Since $\mathbb R^t$ is Zariski dense in $\mathbb C^t$, we can replace $\mathbb C^t$ in Lemma~\ref{lem:polynomial} with $\mathbb R^t$ without destroying the conclusions.

Lemma~\ref{lem:polynomial} is applicable to system \eqref{eq:z-eigenvector-nsym} with the parameter being the tensor $\mathcal A$.
Given a tensor space of size $\otimes^k\mathbb C^n$, the number of E-eigenvectors is
\begin{equation}\label{eq:generic}
\frac{(k-1)^n-1}{k-2}
\end{equation}
if there are only finitely many eigenvectors, which is also the number of E-eigenvectors for a generic tensor \cite{CS-13}.

Thus, the maximal (generic) number of nonsingular solutions of \eqref{eq:z-eigenvector-nsym}, which is the maximal (generic) number of nondegenerate E-eigenvectors, depends on $n$ and $k$ solely. Therefore, for a given tensor space $\otimes^k\mathbb C^n$, the corresponding constant number for a generic tensor with which the system \eqref{eq:z-eigenvector-nsym} has that number of nonsingular solutions in the complex space $\otimes^k\mathbb C^{n}$ can be denoted by $\mathcal N(n,k)$.  Then for every real tensor space $\otimes^k\mathbb R^{n}$, a generic tensor in it has exactly $\mathcal N(n,k)$ nondegenerate E-eigenvectors. Consequently, the number of {nondegenerate} Z-eigenvectors is upper bounded by this $\mathcal N(n,k)$.

Lemma~\ref{lem:polynomial} is for solutions in the algebraic closed field $\mathbb C$. The situation becomes complicated immediately when we switch the interest to real solutions in $\mathbb R^s$, such as Z-eigenvectors of real tensors.
Although $\mathbb R^t$ is Zariski dense in $\mathbb C^t$, the analogue of the number of \textit{real nonsingular solutions}
\[
\mathcal N_{\mathbb R}(\mathbf y):=\#\{\mathbf x\in\mathbb R^s\mid G(\mathbf x,\mathbf y)=\mathbf 0,\ \operatorname{det}(\nabla_{\mathbf x}G(\mathbf x,\mathbf y))\neq 0\}
\]
needs not be a generic constant over $\mathbb R^t$. Of course, {$\mathcal N(n,k)$} is still an upper bound {of} $\mathcal N_{\mathbb R}(\mathbf y)$ for all $\mathbf y\in\mathbb R^t$. It is believed that there are several typical numbers for $\mathcal N_{\mathbb R}(\mathbf y)$ over $\mathbf y\in\mathbb R^t$, as the real tensor rank \cite{L-12}.

Put aside the complicated real case, the exact value for $\mathcal N(n,k)$ (in the complex case) for the system \eqref{eq:z-eigenvector-nsym} is unknown for a given tensor space in the literature. However, Theorem~\ref{thm:sym-orthogonal} gives a lower bound for $\mathcal N(n,k)$. {It is conjectured that $\mathcal N(n,k)$ is equal to the number \eqref{eq:generic}.}

Lemma~\ref{lem:polynomial}, together with Theorem~\ref{thm:critical}, implies the following ``generic version" which is a stronger statement than Theorem~\ref{thm:critical}, since Zariski topology is weaker than the usual Euclidean topology \cite{H-77}.
\begin{theorem}[Generic Nondegenerate Singular Vector Tuples]\label{thm:critical-generic}
For a generic tensor in $\mathbb R^{n_1}\otimes\dots\otimes\mathbb R^{n_k}$, each of its singular vector tuples is nondegenerate.
\end{theorem}

Likewise, we have the following result.
\begin{theorem}[Generic Nondegenerate Z-Eigenvectors]\label{thm:z-critical-generic}
For a generic tensor in $\operatorname{S}(\otimes^k\mathbb R^n)$, each of its Z-eigenvector is nondegenerate.
\end{theorem}

We remark that Theorem~\ref{thm:critical} establishes the nondegeneracy property for all tensors except a set with Lebesgue measure zero. Theorems~\ref{thm:critical-generic} and \ref{thm:z-critical-generic} sharpen it to for all tensors except an algebraic variety with strictly lower dimension than that of the ambient space. The latter set definitely has Lebesgue measure zero. The merit for this refinement also lies in further derivations for determinantal characterizations for tensors having the nondegeneracy property.
\section{Eigenvalues and Eigenvectors}\label{sec:eigenvector}
In this section, we consider the eigenvalues and eigenvectors introduced by Qi \cite{Q-05}.
Given a vector $\mathbf x\in\mathbb C^n$, we denote by $\mathbf x^{[k-1]}$ a vector in $\mathbb C^n$ with the components being
$x_i^{k-1}$ for all $i\in\{1,\dots,n\}$.
\begin{definition}\label{def:eigenvalue}
Given a tensor $\mathcal A\in\otimes^k\mathbb C^n$, if a nonzero vector $\mathbf x$ together with a number $\lambda\in\mathbb C$ satisfies the following equations
\begin{equation}\label{eq:eigenvalue}
\mathcal A\mathbf x^{k-1}=\lambda\mathbf x^{[k-1]},
\end{equation}
then $\lambda$ is an eigenvalue of $\mathcal A$ and $\mathbf x$ a corresponding eigenvector of $\mathcal A$.
\end{definition}

By the theory of determinant of tensors, the number of eigenvalues (counted with multiplicities) is equal to $n(k-1)^{n-1}$ for any given tensor $\mathcal A\in \otimes^k\mathbb C^n$ \cite{HHLQ-13}. Let $\sigma(\mathcal A)$ be the set of all eigenvalues of $\mathcal A$. The {set of eigenvalues of a tensor has} a beautiful symmetric structure in certain cases, which is connected with the underlying zero components pattern of the tensor \cite{H-20}.
However, for a given eigenvalue $\lambda\in\sigma(\mathcal A)$, the set of {the corresponding eigenvectors $V(\lambda)$ (adding the zero vector)} is not a linear subspace of $\mathbb C^n$ any more. It is an \textit{eigenvariety} \cite{HY-16}. In general, the eigenvariety is rather complicated. While, for a generic tensor, it is much clearer. The following result is \cite[Lemma~6.1]{HY-16}.

\begin{lemma}[Unique Eigenvector]\label{lem:unique-eigenvector}
Let tensor $\mathcal A\in\otimes^k\mathbb C^n$ be generic. Then
$V(\lambda)$ has dimension one and is irreducible for all $\lambda\in\sigma(\mathcal A)$, i.e., $\mathcal A$ has a unique (up to scaling) eigenvector for every $\lambda\in\sigma(\mathcal A)$.
\end{lemma}

Lemma~\ref{lem:unique-eigenvector} is also true for generalized tensor eigenvectors, see \cite{FNZ-18}.
\begin{definition}\label{def:nondegenerate}
An eigenvector $\mathbf x$ of $\mathcal A$ is nondegenerate, if at the corresponding eigenvalue the Jacobian matrix of the system \eqref{eq:eigenvalue} has rank $n-1$, or equivalently, it is only singular along the eigenvector $\mathbf x$.
\end{definition}
In the matrix case, we see that an eigenvector is nondegenerate if the corresponding eigenvalue is simple, and a generic matrix has its all eigenvectors being nondegenerate \cite{HJ-85}.

\begin{theorem}\label{thm:nondegenerate}
Let tensor $\mathcal A\in\otimes^k\mathbb C^n$ be generic, then each of its eigenvectors is nondegenerate.
\end{theorem}

\begin{proof}
Let $\mathcal A\in\otimes^k\mathbb C^n$ be generic and $(\lambda,\mathbf x)$ an eigenpair of $\mathcal A$. Then by Lemma~\ref{lem:unique-eigenvector}, we have
\[
V(\lambda)=\{\mathbf y\in\mathbb C^n\colon\mathcal A\mathbf y^{k-1}=\lambda\mathbf y^{[k-1]}\}=\mathbb C\mathbf x.
\]
Being an irreducible smooth variety (actually a point in the projective space or a line in the affine space), the dimension of the eigenvariety $V(\lambda)$ is one, which is equal to the corank of the Jacobian matrix of the defining equations \eqref{eq:eigenvalue} \cite{H-77}.  Consequently, the conclusion follows.
\end{proof}

Motivated by the case for matrices, we conjecture that if the \textit{algebraic multiplicity} of an eigenvalue $\lambda\in\sigma(\mathcal A)$ is one, then a corresponding eigenvector is nondegenerate. While, this is open at present because there lacks an analogue inequality between the algebraic multiplicity and \textit{geometric multiplicity} of an eigenvalue for a tensor as that for a matrix at present. We refer to \cite{HY-16} and references herein for more details.
\section{Conclusions}\label{sec:final}
This short article addressed an interesting as well as important issue for eigenvectors and singular vector tuples of tensors--the nondegeneracy. The nondegeneracy is a foundation of second order analysis for systems of equations \cite{B-99,OR-70}. Compared with the research in the literature \cite{L-05,Q-05,HHLQ-13,HQ-14}, this study moves a further step towards to the second order information of eigenvectors and singular vector tuples of tensors. It is a continuation of the research in \cite{HL-18}. It is certified in \cite{HL-18} that the second order information can be very helpful in furnishing the convergence rate analysis of algorithms for computing them, which is further strengthened in \cite{HY-19} very recently. We hope that the current article can be helpful for designing and analyzing algorithms for computing eigenvectors and singular vector tuples of tensors in the future. Finally, for the purpose of computing eigenvectors of tensors, nonsmooth analysis for the underlying system plays a fundamental role \cite{FP-03}. Actually, it is established in \cite{LQY-13} several interesting semismooth properties for the maximum eigenvalue function for a symmetric tensor. Intrinsic connections between the strong semismoothness of the eigenvalue function and the nondegeneracy for the corresponding eigenvector should be investigated in a forthcoming study.

\section*{Acknowledgements}
This work is partially supported by
National Science Foundation of China (Grant No. 11771328), Young Elite Scientists Sponsorship Program by Tianjin, and the Natural Science Foundation of Zhejiang Province, China (Grant No. LD19A010002).  The author is grateful to Professor Donghui Li at South China Normal University for suggestions.




\begin{thebibliography}{1}
\bibitem{B-99}
Bertsekas DP.  Nonlinear Programming, 2nd ed.  Athena Scientific, Belmont, USA, 1999.

\bibitem{BT-82}
Bott R, Tu LW. Differential Forms in Algebraic Topology. Springer, 1982.

\bibitem{CS-13}
Cartwright D, Sturmfels B. The number of eigenvalues of a tensor. Linear Algebra Appl, 2013, 438: 942--952.

\bibitem{CJ-10}
Comon P, Jutten C. Handbook of Blind Source Separation. Academic Press, Oxford,
2010.

\bibitem{dC-92}
do Carmo MP. Riemannian Geometry. Springer, Berlin, 1992.

\bibitem{EAS-98}
Edelman A, Arias T, Smith ST. The geometry of algorithms with orthogonality
constraints. SIAM J Matrix Anal Appl, 1998, 20: 303--353.

\bibitem{FP-03}
Facchinei F, Pang JS.  Finite-Dimensional Variational Inequalities and Complementarity Problems, vol 1 and vol 2.
Springer-Verlag, New York, 2003.

\bibitem{FNZ-18}
Fan J, Nie J, Zhou A. Tensor eigenvalue complementarity problems. Math Program, 2018, 170: 507--539.

\bibitem{F-92}
Franc A. Etude Alg\'ebrique des Multitableaux: Apports de l'Alg\'ebre Tensorielle. Th\`ese de Doctorat, Sp\'ecialit\'e Statistiques, Univ. de Montpellier II, Montpellier, France, 1992.

\bibitem{GKZ-94}
Gelfand IM, Kapranov MM, Zelevinsky AV. Discriminants, Resultants and Multidimensional Determinants. Birkh\"{a}user, Boston, 1994.

\bibitem{H-12}
Hackbusch W. Tensor Spaces and Numerical Tensor Calculus. Springer, Berlin, 2012.

\bibitem{H-77}
Hartshorne R. Algebraic Geometry. Graduate Texts in Mathematics 52.
Springer, New York, 1977.

\bibitem{HJ-85}
Horn RA, Johnson CR. Matrix Analysis. Cambridge University
Press, New York, 1985.

\bibitem{H-20}
Hu S. Symmetry of eigenvalues of Sylvester matrices
and tensors. Sci  China Math, 2020, 63: 845--872.

\bibitem{HHLQ-13}
Hu S, Huang ZH, Ling C, Qi L. On determinants and eigenvalue theory of tensors.
J Symb Comput, 2013, 50: 508--531.

\bibitem{HL-18}
Hu S, Li G. Convergence rate analysis for the higher order power method in best rank one approximations of tensors. Numer Math, 2018, 140: 993--1031.

\bibitem{HQ-14}
Hu S, Qi L. The E-eigenvectors of tensors. Linear Multilinear A, 2014, 62: 1388--1402.

\bibitem{HY-16}
Hu S, Ye K. Multiplicities of tensor eigenvalues. Commun Math Sci, 2016, 14: 1049--1071.

\bibitem{HY-19}
Hu S, Ye K. Linear convergence of an alternating polar decomposition method for low rank orthogonal tensor approximations. arXiv: 1912.04085.

\bibitem{K-01}
Kolda TG. Orthogonal tensor decompositions. SIAM J Matrix Anal Appl, 2001, 23: 243--255.

\bibitem{L-12}
Landsberg JM. Tensors: Geometry and Applications. AMS, Providence, RI, 2012.

\bibitem{LP-16}
Li G, Pong TK. Calculus of the exponent of Kurdyka-\L ojasiewicz inequality and its applications to linear convergence of first-order methods. Found Comput Math, 2018, 18: 1199--1232.

\bibitem{LQY-13}
Li G, Qi L, Yu G. Semismoothness of the maximum eigenvalue function of a symmetric tensor and its application. Linear Algebra Appl, 2013, 438: 813--833.

\bibitem{L-05}
Lim LH. Singular values and eigenvalues of tensors: a variational approach. In: Proceedings of the 1st IEEE International Workshop on Computational Advances in Multi-Sensor Adaptive Processing, 2005, 129--132.

\bibitem{OR-70}
Ortega JM, Rheinboldt WC. Iterative Solution of Nonlinear Equations in Several Variables. Springer, Berlin, 1970.

\bibitem{Q-05}
Qi L. Eigenvalues of a real supersymmetric tensor. J Symb Comput, 2005, 40: 1302--1324.

\bibitem{R-16}
Robeva E. Orthogonal decomposition of symmetric tensors. SIAM J Matrix Anal Appl, 2016, 37: 86--102.

\bibitem{RS-16}
Robeva E, Seigal A. Singular vectors of orthogonally decomposable tensors. Linear Multilinear A, 2017, 65: 2457--2471.

\bibitem{SW-05}
Sommese AJ, Wampler II CW. The Numerical Solution of Systems of Polynomials
Arising in Engineering and Science. World Scientific, Hackensack, NJ, 2005.

\end{thebibliography}
\end{document}